\documentclass[reqno]{amsart}
\usepackage{amssymb}
\usepackage[dvips]{epsfig}
\usepackage{graphicx}
\usepackage{color}
\usepackage{amsmath}
\usepackage{amssymb}
\usepackage{amsfonts}
\usepackage{amsthm}
\usepackage{bbm}
\usepackage{tikz}

\newcommand{\R}{\mathbb R}

\newcommand{\Z}{\mathbb Z}

\newcommand{\C}{\mathbb C}

\newcommand{\N}{\mathbb{N}}

\newtheorem{thm}{Theorem}[section]
\newtheorem{lem}[thm]{Lemma}

\theoremstyle{remark}
\newtheorem{rem}{\bf Remark}[section]
\theoremstyle{definition}
\newtheorem{defn}[thm]{Definition}

\numberwithin{equation}{section}

\begin{document}

\title[A MSA proof of the power-law localization]{A Multi-scale Analysis Proof of the Power-law Localization for Random Operators on ${\Z}^d$}
\author{Yunfeng Shi}
\address[Y. Shi] {College of Mathematics,
Sichuan University,
Chengdu 610064,
China}
\email{yunfengshi@scu.edu.cn, yunfengshi18@gmail.com}

\date{\today}


\begin{abstract}
In this paper we give a multi-scale analysis proof of the \textit{power-law} localization for random operators on ${\Z}^d$ for \textit{arbitrary} $d\geq1$. 


\end{abstract}

\maketitle






\maketitle
\section{Introduction}
Since the seminal work of Anderson \cite{And58}, the investigation of the localization for noninteracting quasi-particles in random media has attracted great attention  in physics and mathematics community.
In mathematics the first rigorous proof of the localization for random operators was due to Goldsheid-Molchanov-Pastur \cite{GMP77}. They obtained the pure point spectrum for a class of $1D$ continuous random Schr\"odinger operators. In higher dimensions,  Fr\"ohlich-Spencer \cite{FS83} proved, either at high disorder or low energy, the absence of diffusion for some random Schr\"odinger operators by developing the celebrated multi-scale analysis (MSA) method.
Based on the MSA method of Fr\"ohlich-Spencer, \cite{FMSS85,DLS85,SW86} finally obtained the Anderson localization at either high disorder or extreme energy. We should remark that the method of \cite{FS83} was  simplified  and extended by von Dreifus-Klein \cite{vDK89} via introducing a scaling argument. 
Later, the method of \cite{vDK89} was generalized by Klein \cite{Kle94} to prove the Anderson localization for random operators  with the \textit{exponential} long-range hopping. Finally, we want to mention that the MSA method has been strengthened to establish the Anderson localization for random Schr\"odinger operators with Bernoulli  potentials \cite{CKM87,BK05,DS19} \footnote{Very recently, Jitomirskaya-Zhu \cite{JZ19} provided a delicate proof of the  Anderson  localization for the $1D$ random Schr\"odinger operators with Bernoulli potentials using ideas from \cite{Jit99}.
}.

An alternative method for the proof of the localization for random operators, known as the fractional moment method (FMM), was developed by Aizenman-Molchanov \cite{AM93}. This remarkable method also has numerous applications in localization problems \cite{AWB}. As one of its main applications, the FMM was enhanced to prove the first  dynamical localization \cite{Aiz94} for random operators on $\Z^d$. 
Another key application of the FMM in \cite{AM93} is a proof of  the \textit{power-law} localization for  random operators (on $\Z^d$) with polynomially decaying long-range hopping.  Later, the \textit{power-law} localization for $1D$ polynomial long-range hopping random operators was also  established in \cite{JM99} via the method of trace class perturbations.  However, there is simply no MSA proof of the \textit{power-law} localization for {polynomial} long-range hopping random operators, as far as I know. This is the main motivation of the present work. We remark that the Green's functions estimate in the FMM  requires a mild condition (such as the H\"older continuity) on the probability distribution of the potential, while it does not apply to purely singular potentials, such as the Bernoulli ones.  In order to deduce localization using the FMM,  the \textit{absolute continuity} of the measure was even needed in \cite{AM93} so that the Simon-Wolff criterion \cite{SW86} can work.

In this paper we develop a MSA scheme to handle random operators with the \textit{polynomial} long-range hopping and  H\"older continuous distributed potentials (including some \textit{singular continuous} ones).  Although the Bernoulli potentials are not considered in the paper, the suggested method is a promising candidate to be applicable in handling operators with Bernoulli potentials and the \textit{polynomial} long-range hopping. In addition, we think  our formulations in this paper may have  applications in localization problems for other models. In fact, in a forthcoming paper \cite{Shifc} we develop a MSA scheme to study some quasi-periodic operators with \textit{polynomial} long-range hopping.

Our proof is  based essentially on  Fr\"ohlich-Spencer type MSA method \cite{FS83}. In particular, it employs heavily the simplified  MSA method of von Dreifus-Klein \cite{vDK89} (see also \cite{K08}). However, one of the key ingredients in our proof is different from that of \cite{FS83,vDK89} in which the \textit{geometric resolvent identity} was iterated  to obtain the  \textit{exponentially}  decaying of off-diagonal elements of Green's functions. Instead, we directly estimate the \textit{left} inverses of the truncated matrix via information on  small scales Green's functions  and a \textit{prior} $\ell^2$ norm bound of the inverse itself.  This method  was initiated by Kriecherbauer \cite{Kri98} to deal with matrices with \textit{sub-exponentially} decaying (even more general cases) off-diagonal elements, and largely extended by Berti-Bolle \cite{BB13} to study matrices with \textit{polynomially} decaying off-diagonal elements in the context of nonlinear PDEs. 
 Once the Green's functions estimate was established, the proof of the \textit{power-law} localization can be  accomplished with the  Shnol's Theorem. 

The structure of the paper is as follows.  The  \S2 contains our main results on Green's functions estimate (Theorem \ref{msa}) and the \textit{power-law} localization (Theorem \ref{mthm}). The  proof of Theorem \ref{msa} is given in \S3.  In \S4,  the verification of the assumptions (\textbf{P1}) and (\textbf{P2}) in Theorem \ref{msa} is presented. Moreover, the whole MSA argument on Green's functions is also proved there.  In \S5, the proof of Theorem \ref{mthm} is finished. Some useful estimates are included in the appendix.
\section{Main Results}
 Here is the set-up for our main results.
 \subsection{Random Operators with the Polynomial Long-range Hopping}
 Define on $\Z^d$ the {polynomial} long-range hopping $\mathcal{T}$ as
\begin{align}\label{sp2}
\mathcal{T}(m,n)=\left\{\begin{aligned}
&|m-n|^{-r},\ {\rm for}\   m\neq n\ {\rm with}\ m,n\in{\Z}^d,\\
&0,\  {\rm for\ } m=n\in{\Z}^d,
\end{aligned}\right.
\end{align}
where $|n|=\max\limits_{1\leq i\leq d}|n_i|$ and $r>0$.

Let $\{V_\omega(n)\}_{n\in\mathbb{Z}^d}$ be independent identically distributed (\textit{i.i.d}) random variables (with the common probability distribution $\mu$) on some probability space $(\Omega,\mathcal{F}, \mathbb{P})$ ($\mathcal{F}$ a $\sigma$-algebra
on $\Omega$ and $\mathbb{P}$ a probability measure on $(\Omega,\mathcal{F})$).

Let  $\mathrm{supp}(\mu)=\{x:\ \mu(x-\varepsilon,x+\varepsilon)>0\ \mathrm{for}\  \mathrm{any}\  \varepsilon>0\}$ be the support of the common distribution $\mu.$ Throughout this paper we assume $d<r<\infty$ \footnote{
By Schur's test  and the self-adjointness of $\mathcal{T}$, we get (since $r>d$)
\begin{align*}
 \|\mathcal{T}\|\leq \sup_{m\in{\Z}^d}\sum_{n\neq m}|m-n|^{-r}\leq \sum_{n\in{\Z}^d\setminus\{0\}}|n|^{-r}<\infty,
 \end{align*}
where $\|\cdot\|$ is the standard operator norm on $\ell^2({\Z}^d)$.} and
 \begin{itemize}
 \item[$\bullet$]  ${\rm supp}(\mu)$ contains at least two points.
  \item[$\bullet$] ${\rm supp}(\mu)$ is \textit{compact}:  We have  $\mathrm{supp}(\mu)\subset[-{M},{M}]$ for some $M>0$ \footnote{
From \cite{K08}, we have for $\mathbb{P}$ almost all $\omega$,
\begin{align*}
\sup_{n\in{\Z}^d}|V_\omega(n)|\leq M.
\end{align*}
Thus we can assume $\sup\limits_{n\in{\Z}^d}|V_\omega(n)|\leq M$ for all $\omega\in\Omega$.
}.
\end{itemize}

In this paper we study the $dD$ random operators with the {polynomial} long-range hopping
\begin{align}\label{qps}
{H}_{\omega}=\lambda^{-1}\mathcal{T} + V_\omega(n)\delta_{nn'},\ \lambda\geq1,
\end{align}
where $\lambda$ is the coupling constant  for describing the effect of  disorder.

Under the above assumptions, $H_\omega$ is a \textit{bounded self-adjoint} operator on $\ell^2({\Z}^d)$ for each $\omega\in\Omega$.
Denote by $\sigma(H_\omega)$ the  spectrum of $H_\omega$. A well-known result due to Pastur \cite{Pas80} can imply that there exists a set $\Sigma$ (\textit{compact} and \textit{non-random}) such that for $\mathbb{P}$ almost all $\omega$, $\sigma(H_\omega)=\Sigma$.

\subsection{Sobolev Norms of a Matrix}
 Since we are dealing with matrices with polynomially decaying off-diagonal elements, the Sobolev norms introduced by Berti-Bolle \cite{BB13}  are useful.

Fix $s_0>d/2$ (s.t. the Sobolev embedding works).

Let $\langle k\rangle=\max\{1,|k|\}$ if $k\in{\Z}^d$. Define for $u=\{ u(k)\}\in{\C}^{{\Z}^d}$ and $s>0$ the Sobolev norm
\begin{align}\label{us}
\|u\|_s^2=C_0(s_0)\sum_{k\in{\Z}^d}|{u}(k)|^2{\langle k\rangle}^{2s},
\end{align}
where $C_0(s_0)>0$ is fixed so that (for $s\geq s_0$) 
\begin{align*}
\|u_1u_2\|_s\leq \frac{1}{2}\|u_1\|_{s_0}\|u_2\|_s+C(s)\|u_1\|_s\|u_2\|_{s_0}
\end{align*}
with $C(s)>0$, $C(s_0)=1/2$ and $(u_1u_2)(k)=\sum\limits_{k'\in{\Z}^d}{u}_1{(k-k')}{u}_2(k')$.


Let $X_1,X_2\subset{\Z}^d$ be finite sets. Define  $$\mathbf{M}^{X_1}_{X_2}=\left\{\mathcal{M}=(\mathcal{M}(k,k')\in \C)_{k\in X_1,k'\in X_2}\right\}$$ to be the set of all complex matrices with row indexes in $X_1$ and column indexes in $X_2$. If $Y_1\subset X_1,Y_2\subset X_2$, we write $\mathcal{M}^{Y_1}_{Y_2}=(\mathcal{M}(k,k'))_{k\in Y_1,k'\in Y_2}$ for any $\mathcal{M}\in\mathbf{M}^{X_1}_{X_2}$.

\begin{defn}\label{snorm}
Let $\mathcal{M}\in\mathbf{M}^{X_1}_{X_2}$. Define for $s\geq s_0$ the Sobolev norm of $\mathcal{M}$ as
\begin{align*}
\|\mathcal{M}\|_s^2=C_0(s_0)\sum_{k\in X_1-X_2}\left(\sup_{k_1-k_2=k}|\mathcal{M}(k_1,k_2)|\right)^2\langle k\rangle^{2s},
\end{align*}
where $C_0(s_0)>0$ is defined in \eqref{us}.
\end{defn}

\begin{rem}
 From this definition, we have $\|\mathcal{T}\|_{r_1}<\infty$ if $r_1<r-d/2$.
\end{rem}


For more details about Sobolev norms of matrices, we refer to  \cite{BB13}.
\subsection{Green's Functions Estimate}
The estimate on Green's functions plays a key role in spectral theory. In this subsection we  present the first main result on  Green's functions estimate.


For $n\in{\Z}^d$ and $L>0$, define the cube $\Lambda_L(n)=\{k\in{\Z}^d:\ |k-n|\leq L\}$. Moreover,   write $\Lambda_L=\Lambda_L(0)$. The volume of a finite set $\Lambda\subset{\Z}^d$
is defined to be $|\Lambda|=\# \Lambda$. We have $|\Lambda_L(n)|=(2L+1)^d$ ($L\in\N$) for example.

If $\Lambda\subset \Z^d$, denote ${H}_{\Lambda}=R_{\Lambda}{{H}_{\omega}}R_{\Lambda}$, where $R_{\Lambda}$ is the restriction operator. Define the Green's function (if it exists) as
\begin{align*}
G_{\Lambda}(E)=({{H}}_{\Lambda}-E)^{-1},\  E\in\R.
\end{align*}

 Let us introduce  \textbf{good} cubes in ${\Z}^d$.
\begin{defn}
Fix $\tau'>0$, $\delta\in (0,1)$ and $d/2<s_0\leq r_1<r-d/2$. We call $\Lambda_L(n)$ is $(E,\delta)$-\textbf{good} if  $G_{\Lambda_L(n)}(E)$  exists and satisfies
\begin{align*}
\|G_{\Lambda_L(n)}(E)\|_s\leq L^{\tau'+\delta{s}} \ \mathrm{for}\ \forall\  s\in[s_0,r_1].
\end{align*}
Otherwise, we call $\Lambda_L(n)$ is $(E,\delta)$-\textbf{bad}.
We call $\Lambda_{L}(n)$ an $(E,\delta)$-\textbf{good} (resp.  $(E,\delta)$-\textbf{bad}) $L$-cube if it is $(E,\delta)$-\textbf{good} (resp. $(E,\delta)$-\textbf{bad}).
\end{defn}

\begin{rem}\label{rgdec}
Let $\zeta\in (\delta,1)$ and $\tau'-(\zeta-\delta)r_1<0$. Suppose that $\Lambda_L(n)$ is $(E,\delta)$-\textbf{good}. Then we have for $L\geq L_0(\zeta,\tau',\delta,r_1,d)>0$ and $|n'-n''|\geq L/2$,
\begin{align}\label{gdec}
|G_{\Lambda_L(n)}(E)(n',n'')|
&\leq |n'-n''|^{-(1-\zeta)r_1}.
\end{align}
\end{rem}

Assume the following inequalities hold true\footnote{These inequalities will be explained in the proof of the \textbf{Coupling Lemma} in the following.}:
\begin{align}\label{para}
\left\{
\begin{aligned}
&-(1-\delta)r_1+\tau'+2s_0<0,\\
&-\xi r_1+\tau'+\alpha\tau+(3+\delta+4\xi)s_0<0,\\
&\alpha^{-1}(2\tau'+2\alpha\tau+(5+4\xi+2\delta)s_0)+s_0<\tau',\\
\end{aligned}
\right.
\end{align}
where $\alpha,\tau,\tau',r_1>1, \xi>0, s_0>d/2$ and $\delta\in(0,1)$.

Denote by $[x]$ the integer part of  some $x\in\R$. In what follows let $E$ be in an interval $I$ satisfying  $|I|\leq 1$  and $I\cap[-\|\mathcal{T}\|-M,\|\mathcal{T}\|+M]\neq\emptyset$. The main result on Green's functions estimate is as follows.
\begin{thm}[]\label{msa}
Suppose that $1+\xi<\alpha$ and  \eqref{para} holds true. Fix $p>\alpha d, J\in2\mathbb{N}$ and let $L=[l^\alpha]\in \N$ with $l\in \N$.
Then there exists $$\underline{l}_0=\underline{l}_0(\|\mathcal{T}\|_{r_1},M,J,\alpha,\tau,\xi,\tau',\delta, p, r_1,s_0,d)>0$$ such that, if $l\geq \underline{l}_0$,
\begin{itemize}
  \item[(\textbf{P1})]
$\mathbb{P}({\exists}\  E\in I \ {\rm s.t.}\ {\rm both}\  \Lambda_{l}(m)\  {\rm and}\  \Lambda_{l}(n)\ {\rm are\ }  (E,\delta)-{\bf bad})
  \leq l^{-2p}$
for all $|m-n|>2l,$
  \item[(\textbf{P2})] $\mathbb{P}({\rm dist}(\sigma(H_{\Lambda_{L}(m)}),\sigma(H_{\Lambda_{L}(n)}))\leq 2L^{-\tau})\leq L^{-2p}/2$ for all $|m-n|>2L$,
\end{itemize}
then we have
$$\mathbb{P}({\exists}\  E\in I \ {\rm s.t.}\  {\rm both}\ \Lambda_{L}(m)\  {\rm and}\  \Lambda_{L}(n)\ {\rm are\  }(E,\frac{1+\xi}{\alpha})-{\bf bad})
  \leq C(d)L^{-J(\alpha^{-1}p-d)}+L^{-2p}/2$$
for all $|m-n|>2L$.
\end{thm}
\begin{rem}
In this theorem  no regularity assumption on $\mu$ is needed.
Moreover, if we assume further in this theorem $(1+\xi)/\alpha\leq \delta$ and $p>\alpha d+2\alpha p/J$, then the ``propagation of smallness'' for the probability occurs (see Theorem \ref{kthm} in the following for details).
\end{rem}

\subsection{Power-law Localization}
A sufficient condition for the validity of (\textbf{P1}) and (\textbf{P2}) in Theorem \ref{msa} can be derived from some regularity assumption on $\mu$.

Let us recall the H\"older continuity of a distribution defined in \cite{CKM87}.

\begin{defn}[\cite{CKM87}]
We will say a probability measure $\mu$ is H\"older continuous of order $\rho>0$ if
\begin{align}
\frac{1}{\mathcal{K}_\rho(\mu)}=\inf_{\kappa>0}\sup_{0<|a-b|\leq \kappa}|a-b|^{-\rho}\mu([a,b])<\infty.
\end{align}
In this case will call $\mathcal{K}_\rho(\mu)>0$ the order of $\mu$.
\end{defn}
\begin{rem}
\begin{itemize}
\item Let $\mu$ be H\"older continuous of order $\rho$ (i.e., $\mathcal{K}_\rho(\mu)>0$).  Then for any $0<\kappa<\mathcal{K}_\rho(\mu)$, there is some
$\kappa_0=\kappa_0(\kappa,\mu)>0$ so that
\begin{align}\label{mu}
\mu([a,b])\leq \kappa^{-1}|a-b|^\rho\ \mathrm{for}\ 0\leq b-a\leq \kappa_0.
\end{align}

\item If $\mu$ is \textit{absolutely continuous} with a density in $L^q$ with $1<q\leq\infty$, then $\mu$ is H\"older continuous of order $1-1/q$, and $\mathcal{K}_{1-1/q}(\mu)\geq \|\frac{{d}\mu}{{d}x}\|_{L^q}^{-1}$, here ${d}x$ means the Lesbesgue measure on $\R.$
\item There are many \textit{singular continuous} $\mu$ which are H\"older continuous of some order $\rho>0$ \cite{BH80}.
\end{itemize}

\end{rem}

Now we can state  the second main result on the \textit{power-law} localization.
\begin{thm}\label{mthm}
Let ${H}_{\omega}$ be defined by $(\ref{qps})$ with the common distribution $\mu$ being H\"older continuous of order $\rho>0$, i.e., $\mathcal{K}_\rho(\mu)>0$. Let $r\geq\max\{\frac{100d+23\rho d}{\rho}, 331d\}$. Fix any $0<\kappa<\mathcal{K}_\rho(\mu)$. Then there exists $\lambda_0=\lambda_0(\kappa,\mu,\rho,{M},r,d)>0$ such that for  $\lambda\geq\lambda_0$, $H_\omega$
has pure point spectrum  for $\mathbb{P}$ almost all  $\omega\in\Omega$. Moreover, for $\mathbb{P}$ almost all  $\omega\in\Omega$, there exists a complete system of eigenfunctions $\psi_\omega=\{\psi_\omega(n)\}_{n\in{\Z}^d}$  satisfying $|\psi_\omega(n)|\leq  |n|^{-r/600}$ for  $|n|\gg1$.
\end{thm}
\begin{rem}
One may replace $\max\{\frac{100d+23\rho d}{\rho}, 331d\}$ with a smaller one. Actually, if $\mu$ is \textit{absolutely continuous},  it has been proven by Aizenman-Molchanov \cite{AM93} that the \textit{power-law} localization holds for $r>d$ by using the FMM. 
\end{rem}






\section{Proof of Theorem \ref{msa}}


\begin{proof}[\bf Proof of Theorem \ref{msa}]
The proof consists of a deterministic and a probabilistic part.

We begin with the following definition.

\begin{defn}[]
We call a site $n\in\Lambda\subset{\Z}^d$ is $(l,E,\delta)$-\textbf{good} with respect to (\textit{w.r.t}) $\Lambda$ if there exists some $\Lambda_{l}(m)\subset\Lambda$ such that $\Lambda_{l}(m)$ is $(E,\delta)$-\textbf{good} and $n\in \Lambda_{l}(m)$ with ${\rm dist}(n,\Lambda\setminus\Lambda_{l}(m))\geq l/2$. Otherwise, we call $n\in\Lambda\subset{\Z}^d$ is $(l,E,\delta)$-\textbf{bad} \textit{w.r.t} $\Lambda$.
\end{defn}

We then prove a key \textbf{Coupling Lemma}. Recall that $E\in I$ with $|I|\leq 1$ and $I\cap \sigma(H_\omega)\neq\emptyset.$
\begin{lem}[\textbf{Coupling Lemma}]\label{klem}
Let $L=[l^\alpha]$.  Assume that
\begin{itemize}
\item We have \eqref{para} holds true and $1+\xi<\alpha$.
\item We can decompose $\Lambda_L(n)$ into two disjoint subsets $\Lambda_L(n)=B\cup G$ with the following properties: We have
\begin{align*}
B=\bigcup_{1\leq j<\infty}\Omega_{j},
\end{align*}
where for each $j$, ${\rm diam}(\Omega_j)\leq C_\star l^{1+\xi}$ ($C_\star>1$), and for $j\neq j'$, ${\rm dist}(\Omega_j,\Omega_{j'})\geq l^{1+\xi}$. For each $k\in G$, $k$ is $(l,E,\delta)$-{\bf good} w.r.t $\Lambda_{L}(n)$.

\item 
 $$\|G_{\Lambda_L(n)}(E)\|\leq L^{\tau}.$$
\end{itemize}
Then for $$l\geq \underline{l}_0(\|\mathcal{T}\|_{r_1},M,C_\star,\alpha,\tau,\xi,\tau',\delta, r_1,s_0,d)>0,$$ we have $\Lambda_L(n)$ is $(E,\frac{1+\xi}{\alpha})$-{\bf good}.
\end{lem}
\begin{rem}\label{clrem}
The main scheme of the proof is definitely from Berti-Bolle \cite{BB13} in dealing with nonlinear PDEs. 
Since we are also interested in improving the lower bound on $r$, we have to figure out the dependence relations (i.e., \eqref{para}) among various parameters in the iterations. Then the Coupling Lemma (i.e., Proposition 4.1 in \cite{BB13}) of Berti-Bolle may not be used directly here and it needs some small modifications on the proof of Berti-Bolle \cite{BB13}.  It is a key feature that in random operators case the  number of disjoint \textbf{bad} cubes of smaller size contained in a larger cube is fixed and independent of the iteration scales. This  permits us to get \textit{separation} distance of  $l^{1+\xi}$ ($\xi>0$) without increasing the diameter order (of order also $l^{1+\xi}$) of \textbf{bad} cubes clusters (see Lemma \ref{jlem} in the following for details).  As a result, it provides a possible way for improving the lower bound on $r$. 
\end{rem}

\begin{proof}[{\bf Proof of Lemma \ref{klem}}]
The Sobolev norms introduced in \cite{BB13} are convenient to the proof. Below, we collect some useful properties of Sobolev norms for matrices  (see \cite{BB13} for details):
\begin{itemize}
\item ({\bf Interpolation property}): Let $B,C,D$ be finite subsets of ${\Z}^d$ and let $\mathcal{M}_1\in \mathbf{M}^{C}_D,\mathcal{M}_2\in \mathbf{M}_C^B$. Then for any $s\geq s_0$,
\begin{align}\label{ip1}
\|\mathcal{M}_1\mathcal{M}_2\|_s\leq (1/2)\|\mathcal{M}_1\|_{s_0}\|\mathcal{M}_2\|_s+(C(s)/2)\|\mathcal{M}_1\|_s\|\mathcal{M}_2\|_{s_0},
\end{align}
and
\begin{align}
\label{ip2}\|\mathcal{M}_1\mathcal{M}_2\|_{s_0}&\leq \|\mathcal{M}_1\|_{s_0}\|\mathcal{M}_2\|_{s_0},\\
\label{ip3}\|\mathcal{M}_1\mathcal{M}_2\|_s&\leq C(s)\|\mathcal{M}_1\|_s\|\mathcal{M}_2\|_s,
\end{align}
where $C(s)\geq 1$ and $C(s_0)=1$. In particular, if $\mathcal{M}\in \mathbf{M}_B^B$ and $n\geq 1$, then
\begin{align}
\label{ip4}\|\mathcal{M}^n\|_{s_0}&\leq \|\mathcal{M}\|_{s_0}^n, \|\mathcal{M}\|\leq \|\mathcal{M}\|_{s_0},\\
\label{ip5}\|\mathcal{M}^n\|_s&\leq C(s)\|\mathcal{M}\|_{s_0}^{n-1}\|\mathcal{M}\|_s.
\end{align}

\item ({\bf Smoothing property}): Let $\mathcal{M}\in \mathbf{M}^B_C$. Then for $s\geq s'\geq 0$,
\begin{align}
\label{sp1}\mathcal{M}(k,k')=0\  {\rm for}\  |k-k'|<N\Rightarrow \|\mathcal{M}\|_{s'}\leq N^{-(s-s')}\|\mathcal{M}\|_{s},
\end{align}
and for $N\ge N_0(s_0,d)>0$,
\begin{align}\label{sp2}
\mathcal{M}(k,k')=0\  {\rm for}\  |k-k'|>N\Rightarrow \left\{\begin{aligned}
&\|\mathcal{M}\|_{s}\leq N^{s-s'}\|\mathcal{M}\|_{s'},\\
&\|\mathcal{M}\|_{s}\leq N^{s+s_0}\|\mathcal{M}\|.
\end{aligned}\right.
\end{align}
\item ({\bf Columns estimate}): Let $\mathcal{\mathcal{M}}\in \mathbf{M}^B_C$. Then for $s\geq0$,
\begin{align}\label{cl1}
\|\mathcal{M}\|_s\leq C(s_0,d)\max_{k\in C}\|\mathcal{M}_{\{k\}}\|_{s+s_0},
\end{align}
where $\mathcal{M}_{\{k\}}:=(\mathcal{M}(k_1,k))_{k_1\in B}\in \mathbf{M}^B_{\{k\}}$ is a column sub-matrix of $\mathcal{M}$.

\item ({\bf Perturbation argument}): If $\mathcal{M}\in \mathbf{M}^B_C$ has a left inverse  $\mathcal{N}\in \mathbf{M}^C_B$ (i.e., $\mathcal{N}\mathcal{M}=\mathcal{I}$, where $\mathcal{I}$ the identity matrix), then for all $\mathcal{P}\in \mathbf{M}_C^B$ with
$\|\mathcal{P}\|_{s_0}\|\mathcal{N}\|_{s_0}\leq 1/2,$
the matrix $\mathcal{M}+\mathcal{P}$ has a left inverse $\mathcal{N}_\mathcal{P}$ that satisfies
\begin{align}
\label{pl2}\|\mathcal{N}_\mathcal{P}\|_{s_0}&\leq 2\|\mathcal{N}\|_{s_0},\\
\label{pl3}\|\mathcal{N}_\mathcal{P}\|_s&\leq C(s)(\|\mathcal{N}\|_s+\|\mathcal{N}\|_{s_0}^2\|\mathcal{P}\|_s)\ {\rm for\ } s\geq {s_0}.
\end{align}
Moreover, if
$\|\mathcal{P}\|\cdot\|\mathcal{N}\|\leq 1/2,$
then
\begin{align}\label{pl5}
\|\mathcal{N}_\mathcal{P}\|\leq 2\|\mathcal{N}\|.
\end{align}
\end{itemize}

We then turn to the proof of the \textbf{Coupling Lemma}.

Write $\mathcal{A}={H}_{X}-E$ with $X=\Lambda_L(n)$. Let $\mathcal{T}_X=R_X\mathcal{T}R_X$. For $u\in \mathbb{C}^X$ with $X=B\cup G$,  define $u_{G}=R_Gu\in \mathbb{C}^G$ and $u_B=R_Bu\in \mathbb{C}^B$.  Let $h$ be  an arbitrary fixed vector in $\ell^2(X)$ and consider the equation
\begin{align}\label{auh}
\mathcal{A}u=h.
\end{align}

Following \cite{BB13}, we have three steps:

{\bf Step 1: Reduction on \textbf{good} sites}
\begin{lem}[]
Let $l\geq l_0(\tau', \delta, {r_1} , s_0, d)>0$. Then there exist $\mathcal{M}\in{\bf M}_G^X$ and $\mathcal{N}\in{\bf M}_G^B$ satisfying the following:
\begin{align}\label{mnd}
\|\mathcal{M}\|_{s_0}\leq C(s_0,d)l^{\tau'+(1+\delta)s_0},\  \|\mathcal{N}\|_{s_0}\leq C({r_1},s_0,d)\|\mathcal{T}_X\|_{r_1}l^{-(1-\delta){r_1}+\tau'+2s_0}\leq 1/2,
\end{align}
 and for all $s>{s_0}$:
\begin{align}
\label{ms}&\|\mathcal{M}\|_s\leq C(s,s_0,d) l^{2\tau'+(1+2\delta) s_0}(l^{s}\|\mathcal{T}_X\|_{s_0}+\|\mathcal{T}_X\|_{s+s_0}),\\
\label{ns}&\|\mathcal{N}\|_s\leq  C(s,s_0,d)l^{\tau'+\delta s_0}(l^{s}\|\mathcal{T}_X\|_{s_0}+\|\mathcal{T}_X\|_{s+s_0}),
\end{align}
such that
\begin{align}\label{ugub}
u_G=\mathcal{N}u_B+\mathcal{M}h.
\end{align}
\end{lem}
\begin{proof}
Fix $k\in G$. Then there exists some $l$-cube $F_k=\Lambda_l(k_1)$ such that $k\in F_k$, ${\rm dist}(k,X\setminus F_k)\geq l/2$ and $F_k$ is $(E,\delta)$-\textbf{good}. Define $\mathcal{Q}_k=\lambda^{-1} G_{F_k}(E)\mathcal{T}_{F_k}^{X\setminus F_k}\in \mathbf{M}_{F_k}^{X\setminus F_k}$. Since $F_k$ is $(E,\delta)$-\textbf{good} and using the \textbf{Interpolation property} \eqref{ip3}, we obtain (since $\lambda\geq 1$)
\begin{align}\label{qr}
\|\mathcal{Q}_k\|_{r_1}\leq C(r_1)\|G_{F_k}(E)\|_{r_1}\|\mathcal{T}_X\|_{r_1}\leq C({r_1})\|\mathcal{T}_X\|_{r_1}l^{\tau'+\delta {r_1}}.
\end{align}
By the \textbf{Interpolation property} \eqref{ip1} and the \textbf{Smoothing property} \eqref{sp2},  for $s\geq s_0$ we have  (if $|k-k'|>2l$, then $G_{F_k}(E)(k',k)=0$)
\begin{align}
\nonumber\|\mathcal{Q}_k\|_{s+s_0}&\leq C(s)(\|G_{F_k}(E)\|_{s+s_0}\|\mathcal{T}_X\|_{s_0}+\|G_{F_k}(E)\|_{s_0}\|\mathcal{T}_X\|_{s+s_0})\\
\nonumber&\leq C(s)((2l)^{s}\|G_{F_k}(E)\|_{s_0}\|\mathcal{T}_X\|_{s_0}+l^{\tau'+\delta s_0}\|\mathcal{T}_X\|_{s+s_0})\\
\label{qd}&\leq C(s,d)l^{\tau'+\delta s_0}(l^{s}\|\mathcal{T}_X\|_{s_0}+\|\mathcal{T}_X\|_{s+s_0}).
\end{align}

We now vary $k\in G$. Define the following operators\footnote{Both  $\Gamma$ and $\mathcal{L}$ are globally well-defined since we define the operators ``{column by column}''. More precisely, we have shown $G\subset\bigcup\limits_{k\in G}F_k$, and $k\in F_k$ for each $k\in G$. From the definition,  we have {$\Gamma(\cdot,k)$ and $\mathcal{L}(\cdot,k)$ come from that of ${G}_{F_k}(\cdot,k)$}. One may argue that it is likely that there exists $(m,m')$ such that $m'\in F_{k}\cap F_{k'}\neq \emptyset$ for some $k'\neq k$. As a result, it is likely that ${G}_{F_{k}}(m,m')\neq {G}_{F_{k'}}(m,m')$ and $\mathcal{L}(m,m')$ (or $\Gamma(m,m')$) is not uniquely defined! In fact, this is not the case since our definition of $\mathcal{L}(m,m')$ (or $\Gamma(m,m')$) comes from the column  ${G}_{F_{m'}}(\cdot,m')$ rather than that ${G}_{F_k}(\cdot,m')$ or ${G}_{F_{k'}}(\cdot,m')$.}:
\begin{equation*}
\Gamma (k',k)=\left\{\begin{aligned}
&0,\ {\rm for\ } k'\in F_k,\\
&\mathcal{Q}_k(k',k), \ \mathrm{for}\ k'\in X\setminus F_k,
\end{aligned}\right.
\end{equation*}
and
\begin{equation*}
\mathcal{L} (k',k)=\left\{\begin{aligned}
&G_{F_k}(E)(k',k),\ {\rm for\ } k'\in F_k,\\
&0, \ \mathrm{for}\ k'\in X\setminus F_k.
\end{aligned}\right.
\end{equation*}
From \eqref{auh}, we have
\begin{align}\label{ug}
u_G+\Gamma u=\mathcal{L}h.
\end{align}
We estimate $\Gamma\in \mathbf{M}_G^X$. Fix $k\in G$. Note that if $k'\in X\setminus F_k$, then $|k-k'|\geq l/2$. 
This implies $\Gamma_{\{k\}}(k',k)=0$
for $|k'-k|<l/2$. By  the \textbf{Columns estimate}  \eqref{cl1},   \eqref{qr} and the \textbf{Smoothing property} \eqref{sp1}, we obtain
\begin{align}
\nonumber\|\Gamma\|_{s_0}&\leq C(s_0,d)\sup_{k\in G}\|\Gamma_{\{k\}}\|_{2s_0}\\
\nonumber&\leq C(s_0,d) \sup_{k\in G}(l/2)^{-{r_1}+2s_0}\|\Gamma_{\{k\}}\|_{{r_1}}\\
\nonumber&\leq C(s_0,d) \sup_{k\in G}(l/2)^{-{r_1}+2s_0}\|\mathcal{Q}_{k}\|_{{r_1}}\\
\label{gmd}&\leq C(r_1,s_0,d)\|\mathcal{T}_X\|_{r_1}l^{-(1-\delta){r_1}+\tau'+2s_0}.
\end{align}
Similarly, for $s\geq s_0$, we obtain by recalling \eqref{qd}
\begin{align}
\nonumber\|\Gamma\|_s&\leq C(s_0,d)\sup_{k\in G}\|\Gamma_{\{k\}}\|_{s+s_0}\\
\nonumber&\leq C(s_0,d) \sup_{k\in G}\|\mathcal{Q}_{k}\|_{s+s_0}\\
\label{gms}&\leq C(s,s_0,d)l^{\tau'+\delta s_0}(l^{s}\|\mathcal{T}_X\|_{s_0}+\|\mathcal{T}_X\|_{s+s_0}).
\end{align}
We then estimate $\mathcal{L}\in \mathbf{M}_G^X$. Fix $k\in G$. By the definition of $F_k$, if $|k'-k|>2l$, then $k'\notin F_k$. This implies $\mathcal{L}_{\{k\}}(k',k)=0$ for $|k'-k|>2l$.
By the \textbf{Columns estimate} \eqref{cl1} and the \textbf{Smoothing property} \eqref{sp2},
we have for $s\geq 0$,
\begin{align}
\nonumber\|\mathcal{L}\|_{s+s_0}&\leq C(s_0,d)\sup_{k\in G}\|\mathcal{L}_{\{k\}}\|_{s+2s_0}\\
\nonumber&\leq C(s_0,d) \sup_{k\in G}(2l)^{s+s_0}\|\mathcal{L}_{\{k\}}\|_{s_0}\\
\nonumber&\leq C(s,s_0,d) \sup_{k\in G}l^{s+s_0}\|G_{F_k}(E)\|_{s_0}\\
\label{lsd}&\leq C(s,s_0,d)l^{s+\tau'+(1+\delta)s_0}.
\end{align}

Notice that we have ${-(1-\delta){r_1}+\tau'+2s_0<0}$. Thus for $${l\geq l_0(\tau',\delta,{r_1},s_0,d)>0},$$
 we have since \eqref{gmd} $\|\Gamma\|_{s_0}\leq 1/2$. Recalling the \textbf{Perturbation argument} \eqref{pl2}--\eqref{pl3}, we have that $\mathcal{I}+\Gamma_G^G$ is invertible and satisfies
\begin{align}
\label{ggd}\|(\mathcal{I}+\Gamma_G^G)^{-1}\|_{s_0}&\leq 2,\\
\label{ggs}\|(\mathcal{I}+\Gamma_G^G)^{-1}\|_s&\leq C(s) \|\Gamma\|_s\ {\rm for}\ s\geq s_0.
\end{align}
From \eqref{ug}, we have
\begin{align*}
u_G=-(\mathcal{I}+\Gamma_G^G)^{-1}\Gamma_G^{B}u_B+(\mathcal{I}+\Gamma_G^G)^{-1}\mathcal{L}h
\end{align*}
and then
\begin{align}\label{mnl}
\mathcal{N}=-(\mathcal{I}+\Gamma_G^G)^{-1}\Gamma_G^{B}, \  \mathcal{M}=(\mathcal{I}+\Gamma_G^G)^{-1}\mathcal{L}.
\end{align}
Recalling the \textbf{Interpolation property} \eqref{ip1} and  since \eqref{ggd}--\eqref{mnl}, we have
\begin{align*}
\|\mathcal{N}\|_{s_0}&\leq \|(\mathcal{I}+\Gamma_G^G)^{-1}\|_{s_0}\|\Gamma\|_{s_0}\leq C({r_1},s_0,d)\|\mathcal{T}_X\|_{r_1}l^{-(1-\delta){r_1}+\tau'+2s_0},\\
\|\mathcal{M}\|_{s_0}&\leq \|(\mathcal{I}+\Gamma_G^G)^{-1}\|_{s_0} \|\mathcal{L}\|_{s_0}\leq C(s_0,d)l^{\tau'+(1+\delta)s_0},
\end{align*}
and for $s\geq s_0$,
\begin{align*}
\nonumber\|\mathcal{N}\|_s&\leq C(s)(\|(\mathcal{I}+\Gamma_G^G)^{-1}\|_s\|\Gamma\|_{s_0}+\|(\mathcal{I}+\Gamma_G^G)^{-1}\|_{s_0}\|\Gamma\|_s)\\
\nonumber&\leq C(s)\|\Gamma\|_s\\
\nonumber&\leq C(s,s_0,d)l^{\tau'+\delta s_0}(l^{s}\|\mathcal{T}_X\|_{s_0}+\|\mathcal{T}_X\|_{s+s_0})\ ({\rm since} \ \eqref{gms}),\\
\|\mathcal{M}\|_s& \leq C(s)(\|(\mathcal{I}+\Gamma_G^G)^{-1}\|_s\|\mathcal{L}\|_{s_0}+\|(\mathcal{I}+\Gamma_G^G)^{-1}\|_{s_0}\|\mathcal{L}\|_s)\\
&\leq C(s)({\|\Gamma\|_s\|\mathcal{L}\|_{s_0}}+\|\Gamma\|_{s_0}\|\mathcal{L}\|_s) \\
&\leq C(s,s_0,d) l^{2\tau'+(1+2\delta) s_0}(l^{s}\|\mathcal{T}_X\|_{s_0}+\|\mathcal{T}_X\|_{s+s_0})\ ({\rm since} \ \eqref{gmd}-\eqref{lsd}).
\end{align*}
\end{proof}

{\bf Step 2: Reduction on \textbf{bad} sites}
\begin{lem} Let $l\geq l_0(\tau',\delta,{r_1},s_0,d)>0$. We have
\begin{align}\label{apub}
\mathcal{A}'u_B=\mathcal{Z}h,
\end{align}
where
\begin{align*}
\mathcal{A}'=\mathcal{A}_X^B+\mathcal{A}_X^G\mathcal{N}\in\mathbf{M}^B_X,\ \mathcal{Z}={\mathcal{I}}-\mathcal{A}_X^G\mathcal{M}\in\mathbf{M}^X_X
\end{align*}
satisfy for $s\geq s_0$,
\begin{align}
\label{apzd1}\|\mathcal{A}'\|_{s_0}&\leq C(M)(1+\|\mathcal{T}_X\|_{s_0}),\\
\label{apzd2} \|\mathcal{Z}\|_{s_0}&\leq C(M,s_0,d)(1+\|\mathcal{T}_X\|_{s_0})l^{\tau'+(1+\delta)s_0},\\
\label{aps}\|\mathcal{A}'\|_s&\leq C(M,s,s_0,d)(1+\|\mathcal{T}_X\|_{s})l^{\tau'+\delta s_0}(l^{s}\|\mathcal{T}_X\|_{s_0}+\|\mathcal{T}_X\|_{s+s_0}),\\
\label{zs}\|\mathcal{Z}\|_s&\leq  C(M,s,s_0,d)(1+\|\mathcal{T}_X\|_{s})l^{2\tau'+(1+2\delta) s_0}(l^{s}\|\mathcal{T}_X\|_{s_0}+\|\mathcal{T}_X\|_{s+s_0}).
\end{align}
Moreover, $(\mathcal{A}^{-1})_B^X$ is a left inverse of $\mathcal{A}'$.
\end{lem}

\begin{proof}
Since $I\cap[-\|\mathcal{T}\|-M,\|\mathcal{T}\|+M]\neq\emptyset$, $\sup_{\omega,n}{|V_\omega(n)|}\leq M$, $|I|\leq 1$ and $\lambda\geq1$, we have for all $E\in I$ and $n\in{\Z}^d$,
\begin{align*}
|V_\omega(n)-E|\leq \|\mathcal{T}\|+2M+1.
\end{align*}
Thus for any $s\geq 0$, we obtain
\begin{align}\label{as}
\|\mathcal{A}\|_s=\|{H}_X-E\|_s\leq \|\lambda^{-1}\mathcal{T}_X\|_s+\|\mathcal{T}\|+2M+1\leq 2(1+\|\mathcal{T}_X\|_s+\|\mathcal{T}\|+M).
\end{align}
From \eqref{mnd}, \eqref{as} and the \textbf{Interpolation property} \eqref{ip1}--\eqref{ip2}, we have
\begin{align*}
&\|\mathcal{A}'\|_{s_0}\leq \|\mathcal{A}\|_{s_0}+\|\mathcal{A}\|_{s_0}\|\mathcal{N}\|_{s_0}\leq C(M)(1+\|\mathcal{T}_X\|_{s_0})\ ({\rm since}\ \|\mathcal{T}_X\|\leq \|\mathcal{T}_X\|_{s_0} ),\\
&\|\mathcal{Z}\|_{s_0}\leq 1+\|\mathcal{A}\|_{s_0}\|\mathcal{M}\|_{s_0}\leq C(M,s_0,d)(1+\|\mathcal{T}_X\|_{s_0})l^{\tau'+(1+\delta)s_0},
\end{align*}
and for $s\geq s_0$,
\begin{align*}
\|\mathcal{A}'\|_s&\leq \|\mathcal{A}\|_s+C(s)(\|\mathcal{A}\|_s\|\mathcal{N}\|_{s_0}+\|\mathcal{A}\|_{s_0}\|\mathcal{N}\|_s)\\
&\leq C(M,s,s_0,d)(1+\|\mathcal{T}_X\|_{s})l^{\tau'+\delta s_0}(l^{s}\|\mathcal{T}_X\|_{s_0}+\|\mathcal{T}_X\|_{s+s_0}),\\
\|\mathcal{Z}\|_s&\leq 1+C(s)(\|\mathcal{A}\|_s\|\mathcal{M}\|_{s_0}+\|\mathcal{A}\|_{s_0}\|\mathcal{M}\|_s)\\
&\leq C(M,s,s_0,d)(1+\|\mathcal{T}_X\|_{s})l^{2\tau'+(1+2\delta) s_0}(l^{s}\|\mathcal{T}_X\|_{s_0}+\|\mathcal{T}_X\|_{s+s_0}).
\end{align*}
It is easy to see $(\mathcal{A}^{-1})_B^X$ is a left inverse of $\mathcal{A}'$.
\end{proof}

\begin{lem}[Left inverse of $\mathcal{A}'$] Let $l\geq l_0(\|\mathcal{T}_X\|_{r_1}, M, C_\star, \tau, \xi, \tau',\delta, {r_1}, s_0, d)>0$. Then $\mathcal{A}'$ has a left inverse $\mathcal{V}$ satisfying for $s\geq d$,
\begin{align}\label{vs0}
& \|\mathcal{V}\|_{s_0}\leq C(C_\star,s_0,d)l^{\alpha\tau+(2+2\xi)s_0},
\end{align}
and for $s>s_0$,
\begin{align}\label{vs}
&\|\mathcal{V}\|_s\leq  C(M,C_\star,s,s_0,d)(1+\|\mathcal{T}_X\|_{s})l^{\tau'+2\alpha\tau+(4+4\xi+\delta)s_0}(l^{(1+\xi)s}\|\mathcal{T}_X\|_{s_0}+\|\mathcal{T}_X\|_{s+s_0}).
\end{align}
\end{lem}

\begin{proof}
The proof is based on the perturbation of left inverses as in \cite{BB13}. Let $\widetilde \Omega_j$ be the $l^{1+\xi}/4$-neighborhood of $\Omega_j$, i.e.,  $\widetilde \Omega_j=\{k\in\Z^d:\ {\rm dist}(k,\Omega_j)\leq l^{1+\xi}/4\}$.
Let $\mathcal{D}\in \mathbf{M}^B_X$ satisfy the following:
\begin{equation*}
\mathcal{D} (k,k')=\left\{\begin{aligned}
&\mathcal{A}'(k,k'),\ {\rm for\ } (k,k')\in \bigcup_{j}(\Omega_j\times\widetilde\Omega_j),\\
&0, \ \mathrm{for}\ (k,k')\notin \bigcup_{j}(\Omega_j\times\widetilde\Omega_j).
\end{aligned}\right.
\end{equation*}

We claim that $\mathcal{D}$ has a left inverse $\mathcal{W}$ satisfying $\|\mathcal{W}\|\leq 2L^{\tau}$. Let $|k-k'|<l^{1+\xi}/4$ and $\mathcal{R}=\mathcal{A}'-\mathcal{D}$.  Since $B=\bigcup_j\Omega_j$, we have  $k\in \Omega_j$ for some $j$, and then  $k'\in \widetilde\Omega_j$, which implies  $\mathcal{R}(k,k')=0$. Then recalling the \textbf{Smoothing property} \eqref{sp1}, we obtain
\begin{align}
\nonumber\|\mathcal{R}\|_{s_0}&\leq (l^{1+\xi}/4)^{-{r_1}+2s_0}\|\mathcal{R}\|_{{r_1}-s_0}\leq (l^{1+\xi}/4)^{-{r_1}+2s_0}\|\mathcal{A}'\|_{{r_1}-s_0}\\
\nonumber&\leq C(M,{r_1},s_0,d)(1+\|\mathcal{T}_X\|_{r_1})\|\mathcal{T}_X\|_{r_1}l^{-\xi {r_1}+\tau'+(1+\delta+2\xi)s_0}\ ({\rm by}\ \eqref{aps})\\
\label{rd}&\leq C(M,{r_1},s_0,d)l^{-\xi {r_1}+\tau'+(1+\delta+2\xi)s_0}.
\end{align}
Thus recalling \eqref{ip4} and the assumption $\|\mathcal{A}^{-1}\|\leq L^\tau$, we have
\begin{align*}
\|\mathcal{R}\|\cdot\|(\mathcal{A}^{-1})_B^X\|&\leq \|\mathcal{R}\|_{s_0}\|{\mathcal{A}}^{-1}\|\\
&\leq C(M,{r_1},s_0,d)l^{-\xi {r_1}+\tau'+(1+\delta+2\xi)s_0}L^{\tau}\\
&\leq  C(M,{r_1},s_0,d)l^{-\xi {r_1}+\tau'+\alpha\tau+(1+\delta+2\xi)s_0}\leq 1/2\ ({\rm since}\ L=[l^\alpha]),
\end{align*}
where in the last inequality we use the fact that {$-\xi {r_1}+\tau'+\alpha\tau+(1+\delta+2\xi)s_0<0$ and $l\geq l_0(M,\alpha,\tau,\xi,\tau',\delta, {r_1},s_0,d)>0$}.
It follows from the \textbf{Perturbation argument} \eqref{pl5} that $\mathcal{D}$ has a left inverse $\mathcal{W}$ satisfying $\|\mathcal{W}\|\leq 2\|\mathcal{A}^{-1}\|\leq 2L^{\tau}$.

From \cite{BB13}, we know that
\begin{equation*}
\mathcal{W}_0 (k,k')=\left\{\begin{aligned}
&\mathcal{W}(k,k'),\ {\rm for\ } (k,k')\in \bigcup_{j}(\Omega_j\times\widetilde\Omega_j),\\
&0, \ \mathrm{for}\ (k,k')\notin \bigcup_{j}(\Omega_j\times\widetilde\Omega_j)
\end{aligned}\right.
\end{equation*}
is a left inverse of $\mathcal{D}$. We then estimate $\|\mathcal{W}_0\|_s$.  Since ${\rm diam}(\widetilde\Omega_j)\leq 2C_\star l^{1+\xi}$, we have $\mathcal{W}_0(k,k')=0$ if $|k-k'|> 2C_\star l^{1+\xi}$. Using the \textbf{Smoothing property }\eqref{sp2} yields for $s\geq 0$,
\begin{align}
\label{w0s}&\|\mathcal{W}_0\|_s\leq C(C_\star,s,s_0,d)l^{(1+\xi)(s+s_0)}\|\mathcal{W}\|\leq C(C_\star,s,s_0,d)l^{(1+\xi)(s+s_0)+\alpha\tau}.
\end{align}

Finally, recall  that $\mathcal{A}'=\mathcal{D}+\mathcal{R}$ and $\mathcal{W}_0$ is a left inverse of $\mathcal{D}$. We have by \eqref{rd} and \eqref{w0s},
\begin{align*}
\|\mathcal{R}\|_{s_0}\|\mathcal{W}_0\|_{s_0}\leq C(M,C_\star,{r_1},s_0,d)l^{-\xi {r_1}+\tau'+\alpha\tau+(3+\delta+4\xi)s_0}\leq 1/2
\end{align*}
since {$-\xi {r_1}+\tau'+\alpha\tau+(3+\delta+4\xi)s_0<0$ and $l\geq l_0(M,C_\star,\alpha,\tau,\xi,\tau',\delta, {r_1},s_0,d)>0$}.  Applying the {\bf Perturbation argument} \eqref{pl2}--\eqref{pl3} again implies that $\mathcal{A}'$ has a left inverse $\mathcal{V}$ satisfying
\begin{align*}
\|\mathcal{V}\|_{s_0}&\leq 2\|\mathcal{W}_0\|_{s_0}\leq C(C_\star,s_0,d)l^{\alpha\tau+(2+2\xi)s_0},\\
\|\mathcal{V}\|_s&\leq C(s)(\|\mathcal{W}_0\|_s+\|\mathcal{W}_0\|_{s_0}^2\|\mathcal{R}\|_s) \ ({\rm by}\ \eqref{pl3})\\
&\leq C(C_\star,s,s_0,d)l^{(1+\xi)(s+s_0)+\alpha\tau}\\
\nonumber&\ \ +C(M,C_\star,s,s_0,d)(1+\|\mathcal{T}_X\|_{s})l^{\tau'+2\alpha\tau+(4+4\xi+\delta)s_0}(l^{s}\|\mathcal{T}_X\|_{s_0}+\|\mathcal{T}_X\|_{s+s_0})\\
&\leq C(M,C_\star,s,s_0,d)(1+\|\mathcal{T}_X\|_{s})l^{\tau'+2\alpha\tau+(4+4\xi+\delta)s_0}(l^{(1+\xi)s}\|\mathcal{T}_X\|_{s_0}+\|\mathcal{T}_X\|_{s+s_0}).
\end{align*}

\end{proof}

{\bf Step 3: Completion of proof}

Combining \eqref{auh}, \eqref{ugub} and \eqref{apub} implies
\begin{align*}
u_G=\mathcal{M}h+\mathcal{N}u_B,u_B=\mathcal{V}\mathcal{Z}h.
\end{align*}
Thus
\begin{align*}
(\mathcal{A}^{-1})_B^X=\mathcal{V}\mathcal{Z}, \  (\mathcal{A}^{-1})_G^X=\mathcal{M}+\mathcal{N}(\mathcal{A}^{-1})^X_B.
\end{align*}
Then for $s\geq s_0$, we can obtain by using the \textbf{Interpolation property} \eqref{ip1} and the \textbf{Smoothing property} \eqref{sp2}
\begin{align*}
\|(\mathcal{A}^{-1})_B^X\|_s&\leq C(s)(\|\mathcal{V}\|_s\|\mathcal{Z}\|_{s_0}+\|\mathcal{V}\|_{s_0}\|\mathcal{Z}\|_s)\\
&\leq C(M,C_\star,s,s_0,d)(1+\|\mathcal{T}_X\|_{s})^2l^{2\tau'+2\alpha\tau+(5+4\xi+2\delta)s_0}(l^{(1+\xi)s}\|\mathcal{T}_X\|_{s_0}+\|\mathcal{T}_X\|_{s+s_0}) \\
&\ \ +C(M,C_\star,s,s_0,d)(1+\|\mathcal{T}_X\|_{s})l^{2\tau'+\alpha\tau+(3+2\delta+2\xi)s_0}(l^{s}\|\mathcal{T}_X\|_{s_0}+\|\mathcal{T}_X\|_{s+s_0})\\
&\ \  ({\rm by}\ \eqref{zs}\ {\rm and}\ \eqref{vs})\\
&\leq C(M,C_\star,s,s_0,d)(1+\|\mathcal{T}_X\|_{s})^2l^{2\tau'+2\alpha\tau+(5+4\xi+2\delta)s_0}(l^{(1+\xi)s}\|\mathcal{T}_X\|_{s_0}+\|\mathcal{T}_X\|_{s+s_0}).
\end{align*}
We obtain the similar bound for $\|(\mathcal{A}^{-1})_G^X\|_s$. Thus for any $s\in[s_0,{r_1}]$, we obtain
\begin{align*}
\|\mathcal{A}^{-1}\|_s&\leq \|(\mathcal{A}^{-1})_B^X\|_s+\|(\mathcal{A}^{-1})_G^X\|_s\\
&  C(M,C_\star,s,s_0,d)(1+\|\mathcal{T}_X\|_{r_1})^2l^{2\tau'+2\alpha\tau+(5+4\xi+2\delta)s_0}(l^{(1+\xi)s}\|\mathcal{T}_X\|_{s_0}+(2L)^{s_0} \|\mathcal{T}_X\|_{{r_1}})\\
&\leq C(M,C_\star,r_1,s_0,d)\|\mathcal{T}_X\|_{{r_1}}^2 L^{\alpha^{-1}(2\tau'+2\alpha\tau+(5+4\xi+2\delta)s_0)+s_0+\alpha^{-1}{(1+\xi)}s}\\
&\leq L^{\tau'+\frac{1+\xi}{\alpha} s},
\end{align*}
where in the last inequality we use the third inequality in \eqref{para} and  $$L>l\geq l_0(\|\mathcal{T}\|_{{r_1}}, M,C_\star,\alpha,\tau,\xi,\tau',\delta, {r_1},s_0,d)>0.$$

This finishes the proof of the \textbf{Coupling Lemma}.

\end{proof}

We are in a position to finish the proof of  Theorem \ref{msa}. 

{\bf Deterministic part}

\begin{lem}\label{jlem}
Fix $J\in 2\mathbb{N}$.  Assume \eqref{para} holds true and $1+\xi<\alpha$.  Assume further that any pairwise disjoint $(E,\delta)$-{\bf bad} $l$-cubes contained in $\Lambda_{L}(n)$ has number at most $J-1$ and $\|G_{\Lambda_L(n)}(E)\|\leq L^{\tau}$. Then  for $$l\geq l_0(M,J,\alpha,\tau,\xi,\tau',\delta, {r_1},s_0,d)>0,$$  we have $\Lambda_L(n)$ is $(E,\frac{1+\xi}{\alpha})$-{\bf good}.
\end{lem}

\begin{proof}[{\bf Proof of Lemma \ref{jlem}}]
The main point here is to obtain the \textit{separation property} of $(E,\delta)$-\textbf{bad} $l$-cubes contained in $\Lambda_L(n)$.

Denote by $\Lambda_{l}(k^{(1)}),\cdots,\Lambda_l(k^{(t_0)})\subset\Lambda_L(n)$ all the $(E,\delta)$-\textbf{bad} $l$-cubes. Obviously,  $t_0\leq (2L+1)^d$. 
We first claim that there exists $\widetilde{Z}=\{m^{(1)},\cdots,m^{(t_1)}\}\subset Z=\{k^{(1)},\cdots,k^{(t_0)}\}$ such that $t_1\leq J-1$, $|m^{(i)}-m^{(j)}|>2l$ (for $ i\neq j$) and
\begin{align}\label{vitali}
\bigcup_{1\leq j\leq t_0}\Lambda_l(k^{(j)})\subset \bigcup_{1\leq j\leq t_1}\Lambda_{3l}(m^{(j)}).
\end{align}
This claim should be compared with the  Vitali covering argument. We prove this claim as follows. We start from $m^{(1)}=k^{(1)}$. Define $Z_1$ to be the set of all $k^{(j)}$ satisfying $|k^{(j)}-m^{(1)}|\leq 2l$. If $Z_1=Z$, then we stop the process, and it is easy to check  \eqref{vitali}  with $t_1=1.$ Otherwise, we have $Z\setminus Z_1\neq \emptyset$ and we can choose a $m^{(2)}\in Z\setminus Z_1$. Similarly,  let $Z_2$ be the set of all $k^{(j)}\in Z\setminus Z_1$ satisfying $|k^{(j)}-m^{(2)}|\leq l$. If $Z_2=Z\setminus Z_1$,  we stop the process and  \eqref{vitali} holds with $t_1=2$. Repeating this process and since $t_0<\infty$, we can obtain \eqref{vitali} for some $t_1\leq t_0.$ From the construction, we must have $|m^{(i)}-m^{(j)}|>2l$ for $i\neq j$, or equivalently $\Lambda_{l}(m^{(i)})\cap\Lambda_l(m^{(j)})=\emptyset$ for $i\neq j.$ Recalling the assumption of Lemma \ref{jlem}, we get $t_1\leq J-1.$ The proof of this claim is finished.

 We now separate further the clusters $\Lambda_{3l}(m^{(1)})\cap\Lambda_L(n) ,\cdots,\Lambda_{3l}(m^{(t_1)})\cap\Lambda_L(n)$. Define a relation $\bowtie$ on $\widetilde Z$  as follows. Letting $k,k'\in \widetilde Z$, we say $k\bowtie k'$ if there is a sequence $k_0,\cdots, k_q\in\widetilde Z$ ($q\geq 1$) satisfying $k_0=k,k_q=k'$ and
    \begin{align}\label{shi1}
    |k_j-k_{j+1}|\leq 2l^{1+\xi}\ {\rm for}\ \forall\ 0\leq j\leq q-1.
    \end{align}
It is easy to see $\bowtie$ is an equivalence relation on $\widetilde Z.$ As a result, we can partition $\widetilde Z$ into disjoint equivalent classes (\textit{w.r.t}  $\bowtie$), say $\pi_1,\cdots,\pi_{t_2}$ with $t_2\leq t_1$. We also have by \eqref{shi1}
\begin{align}
\label{shi2}|k-k'|&\leq 2Jl^{1+\xi}\  {\rm for}\ \forall\  k,k'\in\pi_j,\\
\label{shi3}{\rm dist}(\pi_{i},\pi_j)&>2l^{1+\xi}\ {\rm for}\ i\neq j.
 \end{align}
Correspondingly, we can define
$$ \Omega_j=\bigcup_{y\in\pi_j}(\Lambda_{3l}(y)\cap\Lambda_L(n)).$$
From \eqref{shi2} and \eqref{shi3}, we obtain
\begin{align*}
{\rm diam}(\Omega_j)&\leq 10Jl^{1+\xi}\ {\rm for}\ 1\leq j\leq t_2,\\
{\rm dist}(\Omega_{i},\Omega_j)&>2l^{1+\xi}-10l>l^{1+\xi}\ {\rm for}\ i\neq j.
\end{align*}
Moreover, since $B=\bigcup\limits_{1\leq j\leq t_2}\Omega_j$ contains all the $(E,\delta)$-\textbf{bad} $l$-cubes, it follows that if $\Lambda_l(k)\subset\Lambda_L(n)$ and $\Lambda_l(k)\cap (\Lambda_L(n)\setminus B)\neq \emptyset$, then $\Lambda_l(k)$ must be $(E,\delta)$-\textbf{good}. Let $G=\Lambda_L(n)\setminus B$. Then every $k\in G$ is $(E,\delta)$-\textbf{good} \textit{w.r.t} $\Lambda_L(n)$. Actually, if $k\in G\subset\Lambda_L(n)$, then there exists $\Lambda_l(k')\subset \Lambda_{L}(n)$ with $k\in\Lambda_l(k')$ such that ${\rm dist}(k,\Lambda_L(n)\setminus\Lambda_l(k'))\geq l$. This $\Lambda_l(k')$ must be $(E,\delta)$-\textbf{good} since $k\notin B,$  i.e., $k$ is $(E,\delta)$-\textbf{good} \textit{w.r.t} $\Lambda_L(n).$

Finally, it suffices to apply Lemma \ref{klem} with $B=\bigcup\limits_{j=1}^{t_2}\Omega_j$, $G=\Lambda_L(n)\setminus B$ and $C_\star=10J$.




\end{proof}

{\bf Probabilistic part}


 Fix $m,n$ with $|m-n|>2L$ and write $\Lambda_1=\Lambda_{L}(m), \Lambda_2=\Lambda_{L}(n)$. We define the
following events for $i=1,2$:
\begin{align*}
&{\mathbf{A}}_i:\ \Lambda_{i}\ {\rm is\ } (E,\frac{1+\xi}{\alpha})-{\rm\bf bad},\\
&{\mathbf{B}}_i:  \ {\rm either}\ G_{\Lambda_i}(E) \ {\rm does\  not\  exist\  or}\  \|G_{\Lambda_i}(E)\|\geq L^{\tau},\\
&{\mathbf{C}}_i:\  \Lambda_i\  {\rm contains\  }J\  {\rm\ pairwise\  disjoint\ } (E,\delta)-{\rm \bf bad}\  l-{\rm cubes},\\
&{\mathbf{D}}:\ {\exists}\  E\in I \ {\rm so\  that\ both}\  \Lambda_{1}\  {\rm and}\  \Lambda_{2}\ {\rm are\ } (E,\frac{1+\xi}{\alpha})-{\rm\bf bad}.
\end{align*}
Using  Lemma \ref{jlem} yields
\begin{align}
\nonumber\mathbb{P}(\mathbf{D})&\leq \mathbb{P}\left(\bigcup_{E\in I}({\mathbf{A}}_1\cap{\mathbf{A}}_2)\right)\leq \mathbb{P}\left(\bigcup_{E\in I}\left(({\mathbf{B}}_1 \cup{\mathbf{C}}_1)\cap({\mathbf{B}}_2 \cup{\mathbf{C}}_2\right))\right)\\
\nonumber&\leq \mathbb{P}\left(\bigcup_{E\in I}({\mathbf{B}}_1\cap{\mathbf{B}}_2)\right)+\mathbb{P}\left(\bigcup_{E\in I}({\mathbf{B}}_1\cap{\mathbf{C}}_2)\right)\\
\nonumber&\ \ +\mathbb{P}\left(\bigcup_{E\in I}({\mathbf{C}}_1\cap{\mathbf{B}}_2)\right)+\mathbb{P}\left(\bigcup_{E\in I}({\mathbf{C}}_1\cap{\mathbf{C}}_2)\right)\\
\label{idp}&\leq \mathbb{P}\left(\bigcup_{E\in I}({\mathbf{B}}_1\cap{\mathbf{B}}_2)\right)+3\mathbb{P}\left(\bigcup_{E\in I}{\mathbf{C}}_1\right).
\end{align}
It is easy to see since (\textbf{P1})
\begin{align}\label{mp}
\mathbb{P}\left(\bigcup_{E\in I}{\mathbf{C}}_1\right)\leq C(d)L^{Jd}(l^{-2p})^{J/2}\leq C(d)L^{-J(\alpha^{-1}p-d)}.
\end{align}
We then estimate the first term in \eqref{idp}. By  (\textbf{P2}), we obtain
\begin{align}
 \nonumber\mathbb{P}\left(\bigcup_{E\in I}({\mathbf{B}}_1\cap{\mathbf{B}}_2)\right)&\leq  \mathbb{P}\left({\rm dist}(\sigma(H_{\Lambda_1}),\sigma(H_{\Lambda_2}))\leq 2L^{-\tau}\right)\\
 \label{l2p}&\leq{L^{-2p}}/{2}.
\end{align}
Combining \eqref{idp}, \eqref{mp} and \eqref{l2p}, we have $\mathbb{P}({\mathbf{D}})\leq C(d)L^{-J(\alpha^{-1}p-d)}+L^{-2p}/2$.

This concludes the proof.
\end{proof}

\section{Validity of (\textbf{P1}) and (\textbf{P2})}
In this section we will verify  the validity of (\textbf{P1}) and (\textbf{P2}) in Theorem \ref{msa}.  As a consequence, we prove a complete MSA argument on Green's functions estimate. The regularity of $\mu$ plays an essential role here. 
\begin{thm}\label{ine}
Let $\mu$ be H\"older continuous of order $\rho>0$ (i.e., $\mathcal{K}_\rho(\mu)>0$). Fix $0<\kappa<\mathcal{K}_\rho(\mu)$, $E_0\in\R$ and $\tau'>(p+d)/\rho$. Then there exists $$\underline{L}_0=\underline{L}_0(\kappa,\mu,\rho,\tau',p,{r_1},s_0,d)>0$$ such that the following holds: if $L_0\geq \underline{L}_0$, then there is some $\lambda_0=\lambda_0(L_0,\kappa,\rho,p,s_0,d)>0$ and $\eta=\eta(L_0,\kappa,\rho,p,d)>0$ so that for $\lambda\geq \lambda_0$, we have
\begin{align*}
 \mathbb{P}({\ \exists}\  E\in [E_0-\eta,E_0+\eta] \ {\rm s.t.\ }\ {\rm both} \ \Lambda_{L_0}(m)\  {\rm and}\  \Lambda_{L_0}(n)\ {\rm are\ }\  (E,\delta)-{\rm\bf bad})
  \leq L_0^{-2p}
  \end{align*}
for all $|m-n|>2L_0$.
\end{thm}
\begin{rem}
We will see in the proof $\lambda_0\sim{L_0^{(p+d)/\rho}}\kappa^{-1/\rho}$ and $\eta\sim {L_0^{-(p+d)/\rho}}\kappa^{1/\rho}$. In addition, $\lambda_0$ and $\eta$ are independent of $E_0$.
\end{rem}

\begin{proof}
Define the event
\begin{align*}
{\mathbf{R}}_n(\varepsilon):\ |V_\omega(k)-E_0|\leq \varepsilon\ {\rm \  for\ some }\ k\in\Lambda_{L_0}(n), 
\end{align*}
where $\varepsilon\in (0,1)$ will be specified below. Then by \eqref{mu}, we obtain for  $2\varepsilon\leq \kappa_0=\kappa_0(\kappa,\mu)>0$,
\begin{align}
\nonumber\mathbb{P}({\mathbf{R}}_n(\varepsilon))&\leq (2L_0+1)^d\mu\left([E_0-\varepsilon,E_0+\varepsilon]\right)\\
\nonumber&\leq 2^\rho(2L_0+1)^d\kappa^{-1} \varepsilon^\rho\\
\label{Rn}&\leq  L_0^{-p},
\end{align}
which permits us to set
\begin{align*}
\varepsilon=2^{-1}3^{-d/\rho}\kappa^{1/\rho}L_0^{-(p+d)/\rho}.
\end{align*}
In particular, \eqref{Rn} holds for $L_0\geq \underline{L}_0(\kappa,\mu,\rho,p,d)>0$.

Suppose now $\omega\notin {\mathbf{R}}_n(\varepsilon)$. Then for all $|E-E_0|\leq \varepsilon/2$ and $k\in\Lambda_{L_0}(n)$, we have
\begin{align*}
|V_\omega(k)-E|&\geq |V_\omega(k)-E_0|-|E-E_0|\geq \varepsilon/2,
\end{align*}
which permits us to set $\eta=\varepsilon/2$. Moreover,  for $\mathcal{D}=R_{\Lambda_{L_0}(n)}(V_\omega(\cdot)-E)R_{\Lambda_{L_0}(n)}$, we have by Definition \ref{snorm} that $\|\mathcal{D}^{-1}\|_s\leq C(d)/\varepsilon$ for $s\geq s_0$. Notice  that
\begin{align*}
\|\lambda^{-1}\mathcal{T} \mathcal{D}^{-1}\|_{s_0}\leq C(s_0,d)\lambda^{-1}\varepsilon^{-1}\leq 1/2
\end{align*}
if $$\lambda\geq \lambda_0=2C(s_0,d)\varepsilon^{-1}.$$
 We assume $\lambda\geq\lambda_0$. Then  by the \textbf{Perturbation argument} (i.e., \eqref{pl2}--\eqref{pl3}) and $$H_{\Lambda_{L_0}(n)}-E=R_{\Lambda_{L_0}(n)}\lambda^{-1} \mathcal{T} R_{\Lambda_{L_0}(n)}+\mathcal{D},$$ we have
\begin{align*}
\|G_{\Lambda_{L_0}(n)}(E)\|_{s_0}&\leq 2\|\mathcal{D}^{-1}\|_{s_0}\leq C(d)\varepsilon^{-1},
\end{align*}
and for $s\geq s_0$,
\begin{align*}
\|G_{\Lambda_{L_0}(n)}(E)\|_{s}&\leq C(s,d)(\varepsilon^{-1}+\lambda^{-1}\varepsilon^{-2})\\
&\leq C(s,s_0,d)\varepsilon^{-1}\ ({\rm since\ }\lambda\geq \lambda_0\sim\varepsilon^{-1}).
\end{align*}
We restrict $s_0\leq s\leq {r_1}$ in the following. In order to show $\Lambda_{L_0}(n)$ is $(E,\delta)$-\textbf{good}, it suffices to let
\begin{align}\label{ll0}
C({r_1},s_0,d)\varepsilon^{-1}=C(\rho,{r_1},s_0,d)\kappa^{-1/\rho}L_0^{(p+d)/\rho}\leq L_0^{\tau'},
\end{align}
which indicates we can allow $L_0\geq \underline{L}_0(\kappa,\mu,\rho,\tau',p,{r_1},s_0,d)>0.$  We should remark here \eqref{ll0} makes sense since $\tau'>(p+d)/\rho$.

Finally, for $|m-n|>2L_0$ and $\lambda\geq \lambda_0$, we have by the \textit{i.i.d} assumption of the potentials that
\begin{align*}
 &\mathbb{P}({\ \exists}\  E\in [E_0-\eta,E_0+\eta] \ {\rm s.t.\ }\  {\rm both}\ \Lambda_{L_0}(m)\  {\rm and}\  \Lambda_{L_0}(n)\ {\rm are\ }\  (E,\delta)-{\rm\bf bad})\\
  &\leq \mathbb{P}(\mathbf{R}_m(\varepsilon))\mathbb{P}(\mathbf{R}_n(\varepsilon))\\
  &\leq L_0^{-2p}\ ({\rm by\ }\eqref{Rn}).
  \end{align*}
\end{proof}


We then turn to the verification of (\textbf{P2}). This will follow from an argument of Carmona-Klein-Martinelli \cite{CKM87}.
\begin{lem}\label{we}
 Let $\mu$ be H\"older continuous of order $\rho>0$ (i.e., $\mathcal{K}_\rho(\mu)>0$). Then for any $0<\kappa<\mathcal{K}_\rho(\mu)$, we can find $\kappa_0=\kappa_0(\kappa,\mu)>0$ so that
 \begin{align*}
\mathbb{P}({\rm dist}(E,\sigma(H_{\Lambda_L(n)}))\leq \varepsilon) \leq \kappa^{-1} 2^\rho(2L+1)^{d(1+\rho)}\varepsilon^\rho
\end{align*}
 for all $E\in\R$, $n\in\Z^d$ and for all $\varepsilon>0$, $L>0$ with  $\varepsilon(2L+1)^d\leq \kappa_0.$
\end{lem}
\begin{proof}
Notice that  the long-range term $\lambda^{-1} \mathcal{T}$ in our operator is \textit{non-random}. Then the proof  becomes similar to that  in the Schr\"odinger operator case by Carmona-Klein-Martinelli \cite{CKM87}. We omit the details here.
\end{proof}



We can then verify (\textbf{P2}) in Theorem \ref{msa}.

\begin{thm}[{\bf Verification of} (\textbf{P2})]\label{vp2}
Let $\mu$ be H\"older continuous of order $\rho>0$ (i.e., $\mathcal{K}_\rho(\mu)>0$). Fix $0<\kappa<\mathcal{K}_\rho(\mu)$. Then For $L\geq \underline{L}_0(\kappa,\mu,\rho,\tau,p,d)>0$ and
\begin{align}\label{tr}
\tau>(2p+(2+\rho)d)/\rho,
\end{align} we have
 \begin{align*}\mathbb{P}({\rm dist}(\sigma(H_{\Lambda_{L}(m)}),\sigma(H_{\Lambda_{L}(n)}))\leq 2L^{-\tau})\leq L^{-2p}/2\end{align*} for all $|m-n|>2L$.
\end{thm}
\begin{proof}
Apply Lemma \ref{we} with $\varepsilon=2L^{-\tau}$. Then we have by the \textit{i.i.d} assumption of  potentials,  \eqref{tr} and $L\geq \underline{L}_0(\kappa,\mu,\rho,\tau,p,d)>0$ that
\begin{align*}
&\mathbb{P}({\rm dist}(\sigma(H_{\Lambda_{L}(m)}),\sigma(H_{\Lambda_{L}(n)})))\\
&\leq \sum_{E\in \sigma(H_{\Lambda_{L}(m)})}\mathbb{P}({\rm dist}(E,\sigma(H_{\Lambda_L(n)}))\leq 2L^{-\tau})\\
&\leq \kappa^{-1} 4^\rho(2L+1)^{d(2+\rho)}L^{-\rho\tau}\\
&\leq L^{-2p}/2.
\end{align*}
\end{proof}

Finally, we provide a complete MSA argument on Green's functions estimate. 

\begin{thm}\label{kthm}
Let $\mu$ be H\"older continuous of order $\rho>0$ (i.e., $\mathcal{K}_\rho(\mu)>0$). Fix $E_0\in\R$ with $|E_0|\leq 2(\|\mathcal{T}\|+M)$, and assume \eqref{para},  \eqref{tr} hold true. Assume further that $(1+\xi)/\alpha\leq \delta$, and $p>\alpha d+2\alpha p/J$ with $J\in 2\N$. Then for $0<\kappa<\mathcal{K}_\rho(\mu)$, there exists $$\underline{L}_0=\underline{L}_0(\kappa,\mu,\rho,\|\mathcal{T}\|_{r_1},M,J,\alpha,\tau,\xi,\tau',\delta, p,{r_1},s_0,d)>0$$ such that the following holds: For $L_0\geq \underline{L}_0$,  there are some $\lambda_0=\lambda_0(L_0,\kappa,\rho,p,s_0,d)>0$ and $\eta=\eta(L_0,\kappa,\rho,p,d)>0$ so that for $\lambda\geq \lambda_0$ and $k\geq 0$, we have
\begin{align*}
 \mathbb{P}({\ \exists}\  E\in [E_0-\eta,E_0+\eta] \ {\rm s.t.\ }\  {\rm both}\ \Lambda_{L_k}(m)\  {\rm and}\  \Lambda_{L_k}(n)\ {\rm are\ }\  (E,\delta)-{\rm\bf bad})
  \leq L_k^{-2p}
  \end{align*}
for all $|m-n|>2L_k$, where $L_{k+1}=[L_k^{\alpha}]$ and $L_0\geq \underline{L}_0$.
\end{thm}
\begin{rem}
\begin{itemize}
\item[]
\item In this theorem we also have $\lambda_0\sim{L_0^{(p+d)/\rho}}\kappa^{-1/\rho}$ and $\eta\sim {L_0^{-(p+d)/\rho}}\kappa^{1/\rho}$. Usually, to prove the localization we can choose $L_0\sim \underline{L}_0$. The key point of the MSA scheme is that the largeness of disorder (i.e., $\lambda_0$) depends only on the initial scales. The later iteration steps do not increase $\lambda$ further. We also observe  that  $\lambda_0$ and $\eta$ are free from $E_0$.

 \item In order to apply Theorem \ref{msa}, we restrict $|E_0|\leq 2(\|\mathcal{T}\|+M)$ in this theorem. Actually, we have $\sigma(H_\omega)\subset[-\|\mathcal{T}\|-M, \|\mathcal{T}\|+M]$.
\end{itemize}
\end{rem}

\begin{proof}
Let  $\underline{L}_{00}=\underline{L}_{00}(\kappa,\mu,\rho,\tau',p,{r_1},s_0,d)>0$  be given by Theorem \ref{ine}.
We choose
\begin{align}\label{ul0}
\underline{L}_0=\max\{\underline{L}_{00},\underline{l}_0\},
\end{align}
where $\underline{l}_0=\underline{l}_0(\|\mathcal{T}\|_{r_1}, M, J, \alpha, \tau, \xi, \tau', \delta, p, {r_1}, s_0, d)$ is given by Theorem \ref{msa}.

Then applying Theorem \ref{ine} with $L_0\geq\underline{L}_0$, ${\lambda}_0={\lambda}_0(L_0,\kappa,\rho,p,s_0,d)$ and $\eta=\eta(L_0,\kappa,\rho,p,d)$ yields
\begin{align*}
 \mathbb{P}({\ \exists}\  E\in [E_0-\eta,E_0+\eta] \ {\rm s.t.\ }\ {\rm both}\   \Lambda_{L_0}(m)\  {\rm and}\  \Lambda_{L_0}(n)\ {\rm are\ }\  (E,\delta)-{\rm\bf bad})
  \leq L_0^{-2p}
  \end{align*}
for all $|m-n|>2L_0$ and $\lambda\geq \lambda_0$.

Let $L_{k+1}=[L_k^\alpha]$ and $L_0\geq\underline{L}_0$.

Assume for some $k\geq 0$ the following holds:
\begin{align*}
 \mathbb{P}({\ \exists}\  E\in [E_0-\eta,E_0+\eta] \ {\rm s.t.\ }\  {\rm both}\ \Lambda_{L_k}(m)\  {\rm and}\  \Lambda_{L_k}(n)\ {\rm are\ }\  (E,\delta)-{\rm\bf bad})
  \leq L_k^{-2p}
  \end{align*}
for all $|m-n|>2L_k$.
Obviously, we have by \eqref{ul0} that ${L}_k\geq {L}_0\geq \underline{L}_0\geq \underline{l}_0>0$. Then applying Theorem \ref{msa} (with $l=L_k,L=L_{k+1}, I=[E_0-\eta,E_0+\eta]$) and Theorem \ref{vp2} yields
\begin{align*}
 \mathbb{P}({\ \exists}\  E\in [E_0-\eta,E_0+\eta] \ {\rm s.t.\ }\  {\rm both}\ \Lambda_{L_{k+1}}(m)\  {\rm and}\  \Lambda_{L_{k+1}}(n)\ {\rm are\ }\  (E,\delta)-{\rm\bf bad})
  \leq L_{k+1}^{-2p}
  \end{align*}
for all $|m-n|>2L_{k+1}$.

This finishes the proof of the  whole MSA argument.

\end{proof}

\section{Proof of Theorem \ref{mthm}}
Recall the Poisson's identity: Let $\psi=\{\psi(n)\}\in {\C}^{\Z^d}$ satisfy $H_\omega \psi=E\psi$. Assume further $G_{\Lambda}(E)$ exists for some  $\Lambda\subset{\Z}^d$. Then for any $n\in\Lambda$, we have
\begin{align}\label{pi}
\psi(n)=-\sum_{n'\in\Lambda,n''\notin\Lambda}\lambda^{-1} G_\Lambda(E)(n,n')\mathcal{T}(n',n'')\psi(n'').
\end{align}

We then introduce the Shnol's Theorem of \cite{H19} in long-range operator case, which is useful to prove our localization. We begin with the following definition.
\begin{defn}
Let $\varepsilon>0$. An energy $E$ is called an $\varepsilon$-generalized eigenvalue if there exists some $\psi\in{\C}^{\Z^d}$ satisfying $\psi(0)=1, |\psi(n)|\leq C(1+|n|)^{d/2+\varepsilon}$  and $H_\omega \psi=E\psi$. We call  such $\psi$ the $\varepsilon$-generalized eigenfunction.
\end{defn}
The Shnol's Theorem for $H_\omega$ reads

\begin{lem}[\cite{H19}]\label{shn} Let $r-2d>\varepsilon>0$ and let $\mathcal{E}_{\omega,\varepsilon}$ be the set of all $\varepsilon$-generalized eigenvalues of $H_\omega$. Then we have $\mathcal{E}_{\omega,\varepsilon}\subset\sigma(H_\omega),\ \nu_\omega(\sigma(H_\omega)\setminus \mathcal{E}_{\omega,\varepsilon})=0$, where $\nu_\omega$ denotes some complete spectral measure of $H_\omega$.
\end{lem}
\begin{rem}
To prove pure point spectrum of $H_\omega$, it suffices to show  each $E\in\mathcal{E}_{\omega,\varepsilon}$ is indeed an eigenvalue of $H_\omega$. Actually, since all eigenvalues of $H_\omega$ are at most countable, it follows from Lemma \ref{shn} that all spectral measures of $H_\omega$  support on a countable set, and thus are of pure point. In the following we even obtain polynomially decaying of each generalized eigenfunction of $H_\omega$ for $\mathbb{P}$ a.e. $\omega$. This  yields the \textit{power-law} localization.
\end{rem}

In what follows we fix $L_0=\underline{L}_0,\lambda_0,\eta$ and $I=[E_0-\eta,E_0+\eta]$ in Theorem \ref{kthm}.

Recalling Theorem \ref{kthm}, we have for $\lambda\geq\lambda_0$ and $k\geq 0$,
\begin{align}\label{pbe}
\mathbb{P}({\exists}\  E\in I \ {\rm s.t.\ }\  {\rm both} \ \Lambda_{L_k}(m)\  {\rm and}\  \Lambda_{L_k}(n)\ {\rm are\ \ } (E,\delta)-{\rm{\bf bad}})
  \leq L_k^{-2p}
  \end{align}
for all $|m-n|>2L_k$, where $L_{k+1}=[L_{k}^{\alpha}]$ and $L_0\gg1$.

We then prove of our  main result on the \textit{power-law} localization.

\begin{proof}[\bf Proof of Theorem \ref{mthm}]
We choose appropriate parameters satisfying \eqref{para}, \eqref{tr}, $(1+\xi)/\alpha\leq \delta$, and $p>\alpha d+2\alpha p/J$ with $J\in 2\N$. For this purpose, we can set by direct calculation the following
\begin{align*}
\alpha=6, \delta=1/2, \xi=2.
\end{align*}
Let $0<\varepsilon\ll1$ (will be specified later). We define $J_\star=J_\star(d,\varepsilon)$ to be the smallest even integer satisfying $p=6d+\varepsilon$ and $p>6d+\frac{12}{J_\star}p. $
As a consequence, we can set
\begin{align*}
\tau=(14/\rho+1)d+O(\varepsilon+\varepsilon/\rho), s_0=d/2+\varepsilon, \tau'=(42/\rho+11/2)d+O(\varepsilon+\varepsilon/\rho).
\end{align*}
Recalling Remark \ref{rgdec}, we set $\zeta=19/20$. Then for
${r_1}\geq (94/\rho+13)d$, we obtain $\tau'<r_1(19/20-1/2)$ (since $\varepsilon\ll1$). Thus if $\Lambda_L(n)$ is $(E,1/2)$-\textbf{good} and  $L\geq L({r_1},d)>0$, then we have by \eqref{gdec} that
\begin{align}
\label{z1}\|G_{\Lambda_L(n)}(E)\|&\leq L^{(42/\rho+23/4)d+O(\varepsilon+\varepsilon/\rho)},\\
\label{z2}|G_{\Lambda_L(n)}(E)(n',n'')|&\leq | n'-n''|^{-{r_1}/20}\ {\rm for}\ |n'-n''|\geq L/2.
\end{align}

For any $k\geq 0$, we define the set $A_{k+1}=\Lambda_{L_{k+1}}\setminus\Lambda_{2L_k}$ and the event
\begin{align*}{\mathbf{E}}_k: \ \exists\  E\in I \ {\rm s.t.\ both }\  \Lambda_{L_k}\  {\rm and}\  {\Lambda_{L_k}(n)}\ \ {\rm (for\ }{\forall\  n\in A_{k+1})}\ {\rm are\ }\ (E,1/2)-{\rm  {\bf bad}}.\end{align*}
Thus from $p=6d+\varepsilon, \alpha=6$ and \eqref{pbe},
\begin{align*}
\mathbb{P}({\mathbf{E}}_k)&\leq (2L_{k}^6+1)^dL_k^{-2(6d+\varepsilon)}\leq C(d)L_k^{-(6d+2\varepsilon)},\\
\sum_{k\geq 0}\mathbb{P}({\mathbf{E}}_k)&\leq \sum_{k\geq 0} C(d)L_k^{-(6d+2\varepsilon)}<\infty.
\end{align*}
By the Borel-Cantelli Lemma, we have $\mathbb{P}({\mathbf{E}}_k\ {\rm occurs \  infinitely \  often})=0.$
If we set $\Omega_0$ to be the event s.t.  ${\mathbf{E}}_k\ {\rm occurs \  only \ finitely \  often}$, then $\mathbb{P}(\Omega_0)=1.$

Let $E\in I$ be an $\varepsilon_1$-generalized eigenvalue and $\psi$ be its generalized eigenfunction, where $0<\varepsilon_1\ll1$ will be specified later. In particular, $\psi(0)=1$. Suppose now there exist infinitely many $L_k$ so that $\Lambda_{L_k}$ are $(E,1/2)$-\textbf{good}. Then from the Poisson's identity \eqref{pi} and \eqref{z1}--\eqref{z2}, we obtain since ${r_1}\geq (100/\rho+15)d$
\begin{align}
\nonumber1=|\psi(0)|&\leq \sum_{n'\in \Lambda_{L_k}, n''\notin \Lambda_{L_k}}C(d)  |G_{\Lambda_{L_k}}(E)(0,n')|\cdot|n'-n''|^{-r}(1+|n''|)^{d/2+\varepsilon_1}\\
\nonumber&\leq {\rm(I)}+{\rm (II)},
\end{align}
where
\begin{align*}
&{\rm(I)}=\sum_{|n'|\leq L_k/2, |n''|>{L_k}}C(\varepsilon_1,d)L_k^{(42/\rho+23/4)d+O(\varepsilon+\varepsilon/\rho)}(|n''|/2)^{-(100/\rho+15)d}|n''|^{d/2+\varepsilon_1},\\
&{\rm(II)}=\sum_{L_k/2\leq |n'|\leq L_k , |n''|>L_k}C(d)  |n'|^{-(5/\rho+3)d}| n'-n''|^{-(100/\rho+15)d}|n''|^{d/2+\varepsilon_1}.
\end{align*}
For ${\rm(I)}$, we have by \eqref{ldec},
\begin{align*}
{\rm(I)}&\leq C(\varepsilon_1,d)L_k^{(42/\rho+27/4)d+O(\varepsilon+\varepsilon/\rho)}L_k^{-(100/\rho+15-3/2)d/2+O(\varepsilon+\varepsilon/\rho+\varepsilon_1)}\\
&\leq C(\varepsilon_1,d)L_k^{-8d/\rho+O(\varepsilon+\varepsilon/\rho+\varepsilon_1)}\to 0\ ({\rm as \ } L_k\to\infty).
\end{align*}
For ${\rm(II)}$, we have also by \eqref{ldec},
\begin{align*}
{\rm(II)}&\leq C(\varepsilon_1,d)L_k^d\sum_{|n''|>L_k} |n''|^{-(5/\rho+3-1/2)d+O(\varepsilon_1)}\\
&\leq C(\varepsilon_1,d)L_k^{-(5\rho/2+1/4)d+O(\varepsilon_1)}\to 0\ ({\rm as \ } L_k\to\infty).
\end{align*}
This implies that for any $\varepsilon_1$-generalized eigenvalue $E$, there exist only finitely many $L_k$ so that $\Lambda_{L_k}$ are $(E,1/2)$-\textbf{good}.

In the following we fix $\omega\in\Omega_0$.

From the above analysis, for ${r_1}\geq (100/\rho+15)d$,  we have shown there exists $k_0(\omega)> 0$ such that for $k\geq k_0$ all $\Lambda_{L_k}(n)$ with $n\in A_{k+1}$ are $(E,1/2)$-\textbf{good}. We define another set $\widetilde A_{k+1}=\Lambda_{L_{k+1}/10}\setminus \Lambda_{10 L_k}$. Obviously, $\widetilde A_{k+1}\subset A_{k+1}$. We will show for ${r}\geq\max\{(100/\rho+23)d, 331d\}$, $\varepsilon,\varepsilon_1\ll1$  and $k\geq k_1=k_1(\kappa,\mu,\rho,{r},d,\omega)>0$ the following holds true:
 \begin{align}\label{dec}
|\psi(n)|\leq |n|^{-{r}/600}\ {\rm for\ } n\in\widetilde A_{k+1}.
\end{align}
 Once \eqref{dec} was established for all $k\geq k_1$, it follows from $\bigcup_{k\geq k_1}\widetilde A_{k+1}=\{n\in{\Z}^d:\ |n|\geq 10L_{k_1}\}$ that $|\psi(n)|\leq {|n|^{-{r}/600}}$ for all $|n|\geq 10L_{k_1}$. This implies that $H_\omega$ exhibits the \textit{power-law} localization on $I$. In order to finish the proof of Theorem \ref{mthm}, it suffices to cover $[-\|\mathcal{T}\|-M,\|\mathcal{T}\|+M]$ by intervals of length $\eta$.

We let $r=r_1+8d$.

We try to prove \eqref{dec}. Notice that $\omega\in\Omega_0$ and $n\in \widetilde A_{k+1}\subset A_{k+1}$. We know $\Lambda_{L_k}(n)\subset A_{k+1}$ is $(E,1/2)$-\textbf{good}.  Then recalling \eqref{pi} again, we have
\begin{align*}
|\psi(n)|&\leq \sum_{n'\in \Lambda_{L_k}(n), n''\notin \Lambda_{L_k}(n)}C(d)  |G_{\Lambda_{L_k}(n)}(E)(n,n')|\cdot|n'-n''|^{-r}(1+|n''|)^{d/2+\varepsilon_1}\\
&\leq {\rm (III)}+{\rm (IV)},
\end{align*}
where
\begin{align*}
&{\rm (III)}\\
&=\sum_{|n'-n|\leq L_k/2, |n''-n|>{L_k}}C(d)L_k^{(42/\rho+23/4)d+O(\varepsilon+\varepsilon/\rho)}(|n''-n|/2)^{-r}(1+L_{k+1}+|n''-n|)^{d/2+\varepsilon_1},\\
&{\rm (IV)}=\sum_{L_k/2\leq |n'-n|\leq L_k , |n''-n|>{L_k}}C(d)  |n'-n|^{-{r_1}/20}|n'-n''|^{-r}(1+L_{k+1}+|n''-n|)^{d/2+\varepsilon_1}.
\end{align*}
For {\rm (III)}, we have by \eqref{ldec},
\begin{align*}
{\rm (III)}&\leq C(\varepsilon_1,r,d)L_{k+1}^{d/2+\varepsilon_1}L_k^{(42/\rho+27/4)d+O(\varepsilon+\varepsilon/\rho)}\sum_{|n''-n|>L_k} |n''-n|^{-r+d/2+\varepsilon_1}\\
&\leq C(\varepsilon_1,r,d)L_k^{(42/\rho+39/4)d+O(\varepsilon+\varepsilon/\rho+\varepsilon_1)} L_k^{(-r+3d/2+\varepsilon_1)/2}\\
&\leq C(\varepsilon_1,r,d)L_k^{-r/2+(42/\rho+21/2)d+O(\varepsilon+\varepsilon/\rho+\varepsilon_1)}.
\end{align*}
For {\rm (IV)}, we also have  by \eqref{ldec},
\begin{align*}
{\rm (IV)}&\leq C(\varepsilon_1,{r_1},d)L_{k+1}^{d/2+\varepsilon_1}L_k^d\sum_{|n''-n|>L_k} |n''-n|^{-{r_1}/20+d/2+\varepsilon_1}\\
&\leq C(\varepsilon_1,{r_1},d)L_k^{4d+6\varepsilon_1} L_k^{(-{r_1}/20+3d/2+\varepsilon_1)/2}\\
&\leq C(\varepsilon_1,{r_1},d)L_k^{-{r_1}/40+19d/4+7\varepsilon_1}.
\end{align*}
Combining the above estimates and since ${r}\geq\max\{(100/\rho+23)d, 331d\}$, we have
\begin{align*}
|\psi(n)|\leq C(\varepsilon_1,{r_1},d)L_k^{(-{r_1}/40+19d/4+16\varepsilon+7\varepsilon_1)/6} \leq |n|^{-{r}/600},
\end{align*}
where we use $|L_k|\geq |n|^{1/6}\gg1$ for $n\in \widetilde{A}_{k+1}$, and $\varepsilon,\varepsilon_1\ll1$.

\end{proof}

\section*{Acknowledgements}
I would like to thank  Svetlana Jitomirskaya for reading earlier versions of the paper
and her constructive  suggestions.   The author is grateful to Xiaoping Yuan  for his  encouragement.
This work was supported by NNSF of China  grant 11901010.

\appendix


\section{}
We  introduce a useful lemma.
\begin{lem}
Let $L>2$ with $L\in\N$ and $\Theta-d>1$. Then we have
\begin{align}\label{ldec}
\sum_{n\in{\Z}^d:\ |n|\geq L}|n|^{-\Theta}\leq C(\Theta,d)L^{-(\Theta-d)/2},
\end{align}
where $C(\Theta,d)>0$ depends only on $\Theta,d$.
\end{lem}

\begin{proof}
Obviously, $\#\{n\in{\Z}^d:\ |n|=j\}\leq 2dj^{d-1}$ for any $j\in \N$. Thus
\begin{align*}
\sum_{n\in{\Z}^d:\ |n|\geq L}|n|^{-\Theta}&\leq \sum_{j\geq L}\sum_{|n|=j}|n|^{-\Theta}\leq 2d\sum_{j\geq L}j^{-(\Theta-d+1)}\\
&\leq 2d \sum_{l\geq 1}(l+L-1)^{-(\Theta-d+1)}\\
&\leq 2d \sum_{l\geq 1}(2l(L-1))^{-(\Theta-d+1)/2}\ ({\rm for}\ l+(L-1)\geq 2\sqrt{l(L-1)})\\
&\leq {(2L-2)}^{-(\Theta-d+1)/2}2d\sum_{l\geq 1}l^{-(\Theta-d+1)/2}\\
&\leq C(\Theta,d)L^{-(\Theta-d+1)/2},
\end{align*}
where in the last inequality we use $2L-2>L, -(\Theta-d+1)/2<-1$ and
\begin{align*}
C(\Theta,d)=2d\sum_{l\geq 1}l^{-(\Theta-d+1)/2}<\infty.
\end{align*}

\end{proof}

\bibliographystyle{alpha}

 \end{document}